\documentclass[11pt]{amsart} 
\usepackage{amsmath, amsthm, amssymb} 
\usepackage{geometry}
\usepackage{setspace}
\theoremstyle{plain}
\newcounter{thmcount}[section]

\theoremstyle{definition}  % 这个样式默认使用正体
\newtheorem{Definition}[thmcount]{Definition}

\theoremstyle{plain}  % 这个样式使用斜体
\newtheorem{Lemma}[thmcount]{Lemma}          
\newtheorem{Theorem}[thmcount]{Theorem}       
    
\newtheorem{Proposition}[thmcount]{Proposition} 
\newtheorem{Convention}[thmcount]{Convention} 

\title{\textbf{Pure Core Sets of $n \times n$ Matrices over Finite Fields}}
\author{Hongyu Wang AND  Yizhi Zhang}

\newcommand\st{\ \mathrm{s.t.}\ }
\newcommand\GL{\mathrm{GL}}
\newcommand\Gal{\mathrm{Gal}}
\renewcommand\span{\mathrm{span}}
\renewcommand\Im{\mathrm{Im}}

\newcommand\Z{\mathbb{Z}}
\newcommand\Fq{\mathbb{F}_q}

\newcommand\E{\mathbb{E}}
\newcommand\C{\mathcal{C}}
\newcommand\N{\mathcal{N}}
\newcommand\B{\mathcal{B}}

\begin{document}

\maketitle

 \begin{abstract}
 This paper studies the structure of core sets under different similarity classes. We investigate the influence of factors of the minimal polynomial with different degrees on the structure of core sets. When $F$ is a finite field of prime order, we study the upper bound on the size of a non-core set in a similarity class in $M_n(F)$.  We prove that as $|F|$ increases, the proportion of pure core sets among subsets of $M_n(F)$ tends to $1$.
 
 \vspace{1em}
 

\end{abstract}

\section{Introduction}
We first fix some notations: For a field $F$, let $M_n(F)$ denote the ring of $n \times n$ matrices over $F$,
and $\GL(n,F)$ the general linear group consisting of all invertible matrices in $M_n(F)$.
Let $I$ denote the identity matrix. By identifying $a\in F$ with $aI\in M_n(F)$,
we may consider $F\subseteq M_n(F)$.
Unless otherwise specified, vectors in this paper are column vectors, and $F^n$ denotes the $n$-dimensional column vector space over $F$, with its zero subspace denoted by $0$.
For a matrix $A\in M_n(F)$, let $\Im_F A$ denote the image space of the map $F^n\to F^n,x\mapsto Ax$. When clear from context, it is also denoted $\Im A$. Also, denote the minimal polynomial of $A$ by $m_A$.
Let $\C$ denote a similarity class in $M_n(F)$.
Let $\mu_\C$ denote the minimal polynomial of elements in $\C$, also abbreviated $\mu$.
For a root $\alpha$ of $\mu$, let $f_\alpha=\dfrac{\mu}{x-\alpha}$.

Regarding polynomial operations, we assume the indeterminate $x$ is central (i.e., commutes with all coefficients),
and polynomials in $M_n(F)[x]$ follow the right evaluation rule------before evaluating $x$,
the polynomial must be expressed such that $x$ appears only to the right of coefficients.
For example, the product $(Ax+B)(Cx)$ must first be simplified to $ACx^2+BCx$ before evaluation.
For detailed treatment of polynomials over noncommutative rings, see reference [1].

Let $S\subseteq M_n(F)$. This paper focuses on studying the zero ideal of the matrix set $S$, defined as follows:

\begin{Definition}

Define the zero ideal of $S$ as $N(S):=\{f\in M_n(F)[x]\mid f(A)= 0,\forall A\in S\}$.
\end{Definition}

Note that the coefficients of polynomials in $N(S)$ are elements of $M_n(F)$, which are noncommutative.

The central question of this paper is:
For a given $S\subseteq M_n(F)$, is its zero ideal $N(S)$ a two-sided ideal of $M_n(F)[x]$?

\begin{Definition}

A set $S\subseteq M_n(F)$ is called a core set if its zero ideal $N(S)$ is a two-sided ideal of $M_n(F)[x]$;
otherwise, $S$ is called a non-core set.
\end{Definition}

The study of this problem originates from integer-valued polynomials. The core of integer-valued polynomials is to characterize those polynomials that, when evaluated at elements of a specific subset, still lie in the original ring. For a commutative integral domain $D$ with field of fractions $K$, the classical ring of integer-valued polynomials is defined as:
\[
\mathrm{Int}(S, D) := \{ f \in K[x] \mid f(s) \in D,\ \forall s \in S \},
\]
where $S \subseteq D$. When $D = \mathbb{Z}$, the polynomials in $\mathrm{Int}(S, \mathbb{Z})$ are precisely those that yield integer values when evaluated at all integers in $S$.

When considering subsets $S \subseteq M_n(D)$ of the matrix ring, it remains an open problem whether the corresponding set of integer-valued polynomials
\[
\mathrm{Int}(S, M_n(D)) := \left\{ f \in M_n(K)[x] \mid f(A) \in M_n(D),\ \forall A \in S \right\}
\]
forms a ring [2, Question 2.1].

An equivalent proposition is provided in [2, 3]:
$\mathrm{Int}(S, M_n(R))$ is a ring if and only if for all nonzero $a\in R$,
the image of $S$ in $M_n(R/aR)$ is a core set.
Therefore, characterizing subsets $S$ for which $\mathrm{Int}(S,M_n(R))$ is a ring
is equivalent to characterizing core sets in $M_n(R/aR)$,
which is the problem of interest in this paper.

Recently, researchers have studied zero ideals in the general case, with particular attention to $2 \times 2$ and $3 \times 3$ matrices over fields [4, 5, 6]. It is easy to see that for any $n \geq 2$, the matrix ring $M_n(R)$ contains both core subsets and non-core subsets.
Less trivially, if $S \subseteq M_n(R)$ consists of scalar matrices,
then the evaluation behavior of polynomials at elements of $S$ coincides with that in the ordinary commutative case,
and in this case $S$ is always a core set.
More generally, the entire similarity class of a matrix $A$ is always a core set [4, Proposition 3.1].
A complete algorithmic characterization of core sets in $M_2(\mathbb{F}_q)$ is provided in [4], while some partial results on core sets in $M_n(\mathbb{F}_q)$ for $n \geq 3$ are given in [4, 6]. In this paper, we make use of the results from [4, 5, 6].

\begin{Proposition}\textup{([4, Proposition 3.1,3.4])}
Let $S\subseteq M_n(F)$, then

(1) $N(S)$ is a left ideal of $M_n(F)[x]$.

(2) Let $I$ be an index set, $\{S_i\}_{i\in I}$ a family of core sets, then $\bigcup\limits_{i\in I}S_i$ is a core set.

(3) Let $a\in F$, then $S$ is a core set $\iff S+a$ is a core set.

(4) Let $0\neq a\in F$, then $S$ is a core set $\iff aS$ is a core set.
\end{Proposition}

\begin{Definition}([4, Definition 4.3])
For $\varnothing\neq S\subseteq M_n(F)$,
if the polynomial family $\{m_A|A\in S\}$ has a monic least common multiple,
then the monic least common multiple of $\{m_A|A\in S\}$ is called the minimal polynomial of $S$,
denoted $\mu_S$, and $S$ is said to be polynomial finite.
\end{Definition}

\begin{Proposition}\textup{([4] Theorem 4.4, Corollary 4.6)}
Assume $S\subseteq M_n(F)$ is polynomial finite, then

(1) $\mu_S\in N(S)$.

(2) $S$ is a core set $\iff N(S)=(\mu_S)$.

(3) $S$ is a non-core set if and only if $N(S)$ contains a nonzero polynomial of degree lower than $\mu_S$.
\end{Proposition}

\begin{Definition}([5] Definition 1.2)
If for every similarity class $\C$, the intersection $S\cap\C$ is a core set, then $S$ is called a pure core set.
\end{Definition}

Previous researchers emphasize that when $n$ is small, similarity classes and minimal polynomial classes are identical or nearly identical [4, 5, 6]. However, this is not always the case. Therefore, our definition of pure core sets here differs slightly from [5], adopting the finer partition of similarity classes. As we will see later, such a partition is necessary.

The question of how common core sets are was raised in [5]. The goal of this line of inquiry is to bound the probability that a randomly chosen subset of $M_n(F)$ is a core set. This question was addressed for $M_2(\mathbb{F}_q)$ in [5], where the exact count of core sets in each similarity class was determined.

We extend this result, building on [5],
by achieving this goal for $M_n(\mathbb{F}_q)$ where $q$ is prime,
determining the structure of the distribution of pure core sets within each similarity class of $M_n(\mathbb{F}_q)$,
and providing an algorithm for computing the number of core sets in a similatity class. Moreover, using this, we prove that:
as $q \to \infty$,
the probability that a randomly chosen subset of $M_n(\mathbb{F}_q)$ is a core set tends to 1.
Therefore, in the asymptotic sense as $q$ tends to infinity, almost all subsets of $M_n(\mathbb{F}_q)$ are core sets.

The following is a summary of the main results of this paper.

This paper provides the structure of the distribution of pure core sets in $M_n(\Fq)$ where $q$ is prime.
Building on [4], to address the problem of core set determination under finer similarity class partition, this paper first introduces sets $\mathcal{N}(A,g)$ and $\mathcal{B}(A,g)$ associated with a single matrix $A$ and a polynomial factor $g$. Based on these two types of sets, this paper establishes necessary and sufficient conditions for a polynomial finite set $S$ to be a core set, breaking down the problem of pure core sets within a similarity class into each polynomial factor. Then, based on the degree of the factors, the cases are divided into linear factors and higher-degree irreducible factors.

For the case of linear factors $x-a$ in the minimal polynomial, this paper further introduces $\mathcal{B}(A, f_a)$, establishing its connection with the orthogonal complement of the matrix image space. From this, a linear space criterion for core sets is derived:

\begin{Theorem}

Let $\C\subseteq M_n(F),\varnothing\neq S\subseteq \C$.

(1) If $S$ is a core set, then for any $a\in F$ satisfying $x-a\mid\mu$,
we have $\sum\limits_{A\in S}\Im f_a(A) = F^n$.

(2) When $\mu$ has only linear irreducible factors, the converse of the above statement also holds.
\end{Theorem}

Based on this criterion, this paper further partitions the similarity class $\mathcal{C}$ into subclasses according to the equivalence of $\Im f_a(A)$, proving that the subclasses have equal sizes and the number of subclasses is given by subspace counting formulas.

For the case of higher-degree irreducible factors $m$ in the minimal polynomial, this paper utilizes the perfectness of finite fields $\Fq$,
introducing the splitting field $K$ of $m$ over $\Fq$.
By exploiting the transitivity of the Galois group action, this paper proves that the intersection of $\mathcal{B}(A, \frac{\mu}{m})$ over $\mathbb{F}_q$ is zero if and only if the intersection of $\mathcal{B}(A, f_{\alpha_i})$ over all roots $\alpha_i$ in $K$ is zero. This equivalence, combined with the conclusion for linear factors, ultimately transforms the core set determination for higher-degree factors into the problem of deciding whether the sum of $\Im f_{\alpha}(A)$ equals $K^n$.
Based on the above results for individual factor analysis,
an upper bound on the size of non core sets in a similarity class can be given:

\begin{Theorem}

Let $\C\subseteq M_n(\Fq)$ and $S\subseteq \C$ be a non-core set, then $|S|\leq\dfrac{4|\C|}{q}$.

\end{Theorem}

Furthermore, combining probabilistic tools to analyze the distribution of pure core sets,
using Chernoff's inequality,
this paper proves:

\begin{Theorem}

Let $S\subseteq M_n(\Fq)$ be a random subset, then $\lim\limits_{q\to\infty}\mathbb{P}(S\text{ is a pure core set})=\lim\limits_{q\to\infty}\mathbb{P}(S\text{ is a core set})=1$.

\end{Theorem}

This asymptotic result indicates that over large finite fields, almost all subsets of $M_n(F)$ are pure core sets.

\noindent\textbf{Acknowledgements:}
We especially thank our advisor Professor Huaiqing Zuo for proposing the problem and providing guidance.
We especially thank Yufan Hua for some inspirations.

\section{Preliminary Propositions}

In what follows, let $F$ be a finite field, and denote $|F|=q$.

This section lists an inequailty lemma and linear algebra preliminary propositions needed for the proofs, most of which are easily verified.

\begin{Lemma}

Let $C=\prod\limits^{\infty}_{i=1}\dfrac{1}{1-q^{-i}}$, then $C\leq 4$.
\end{Lemma}

\begin{proof}

Since $q\geq 2$, it suffices to prove $\prod\limits^{\infty}_{i=1}\dfrac{1}{1-2^{-i}}\leq 4$, i.e., $\sum\limits^{\infty}_{i=1}-\ln(1-2^{-i})\leq\ln 4$.
Construct $h(x)=\ln(1-x)+x+x^2,x\in[0,1)$, then $h'(x)=\dfrac{x(1-2x)}{1-x}$.
When $x\in(0,\dfrac{1}{2})$, $h'(x)>0$, hence for $x\in[0,\dfrac{1}{2}]$ we have $h(x)\geq h(0)=0$.
Therefore, for $x\in[0,\dfrac{1}{2}]$, $-\ln(1-x)\leq x+x^2$.
Thus $\sum\limits^{\infty}_{i=1}-\ln(1-2^{-i})\leq\sum\limits^{\infty}_{i=1}2^{-i}+\sum\limits^{\infty}_{i=1}4^{-i}=\dfrac{4}{3}<\ln 4$.
The original inequality is proved.
\end{proof}

\begin{Definition}

Let $V$ be a subspace of $F^n$, define the orthogonal complement of $V$ as
$$V^\perp:=\{x \in F^n\mid x^T y=0,\ \forall y \in V \}.$$

Let $I$ be an index set, $\{V_i\}_{i\in I}$ a family of subspaces of $F^n$, we denote the sum of this family of subspaces by $\sum\limits_{i\in I}V_i$.
\end{Definition}

\begin{Proposition}

Let $V$ be a subspace of $F^n$, then

(1) $\dim V^\perp=n-\dim V$.

(2) $(V^\perp)^\perp=V$.

(3) Let $I$ be an index set,
$\{V_i\}_{i\in I}$ a family of subspaces of $F^n$,
then $(\bigcap\limits_{i\in I} V_i)^\perp=\sum\limits_{i\in I}V^\perp_i$.
\end{Proposition}

\begin{Definition}

Let $\varnothing\neq S\subseteq M_n(F)$. Consider $R_S=\{v\in F^n\mid v^T\text{ is a row vector of some }A\in S\}$. Then define the subspace $V_S=\span(R_S)$ generated by $R_S$ as the row space of $S$.
\end{Definition}

\begin{Proposition}

(1) If $S\subseteq T\subseteq M_n(F)$, then $V_{S}\subseteq V_{T}$.

(2) Let $I$ be an index set, $\{S_i\}_{i\in I}$ be a family of subsets of $M_n(F)$, then $V_{\bigcap\limits_{i\in I}S_i}\subseteq\bigcap\limits_{i\in I} V_{S_i}$.
\end{Proposition}

\begin{Definition}

Let $K/F$ be a finite extension. For a subspace $V$ of $F^n$, denote the space spanned by the vectors of $V$ in $K^n$ as $V_K$.
\end{Definition}

\begin{Proposition}

Let $V,W$ be subspaces of $F^n$, then

(1) $V_K\cap F^n=V$.

(2) $V\subseteq W\iff V_K\subseteq W_K$.

(3) $(V\cap W)_K=V_K\cap W_K$.
\end{Proposition}

\section{Equivalent Polynomial Characterization of Core Sets}

To investigate the criteria for determining core sets, this section introduces polynomial sets $\N_F(A,g)$ and $\B_F(A,g)$ associated with a single matrix,
analyzes their algebraic structure, and establishes equivalent characterizations of core sets based on this.

\begin{Definition}

Let $A\in M_n(F), g\in F[x], g|m_A$. Define
$$\N_F(A,g)=\{f\in M_n(F)[x]\,\big|\,(fg)(A)=0\},$$
$$\B_F(A,g)=\{f\in\N_F(A,g)\,\big|\,\deg f<\deg m_A-\deg g\}.$$

When no confusion arises, $\N_F(A,g),\B_F(A,g)$ are abbreviated as $\N(A,g),\B(A,g)$ respectively.
\end{Definition}

\begin{Proposition}

Let $A\in M_n(F)$, and let monic polynomial $g\in F[x]$ satisfy $g|m_A$, then

(1) $\N(A,g)$ is a left ideal of $M_n(F)[x]$.

(2) Denote $k=\dfrac{m_A}{g}$, then $k\in\N(A, g)$, hence $\N(A,g)=\{f+hk\,\big|\,f\in\B(A,g),h\in M_n(F)[x]\}$.
\end{Proposition}

Utilizing the properties of $\N(A,g)$ and $\B(A,g)$ above,
we can establish equivalent conditions for a polynomial finite set $S$ to be a core set.

\begin{Proposition}

Assume $S$ is polynomial finite, and $p$ is a prime factor of $\mu_S$ with multiplicity $c$.
Denote $S_p=\{A\in S\,\big|\,p^c|m_A\}$, then the following propositions are equivalent:

(1) $\forall f\in N(S), p^c|f$.

(2) $\bigcap\limits_{A\in S_p}\B(A,\dfrac{m_A}{p^c})=\{0\}$.

(3) $\bigcap\limits_{A\in S_p}\B(A,\dfrac{m_A}{p})=\{0\}$.
\end{Proposition}

\begin{proof}

Denote the negations of (1)(2)(3) as $(\neg 1)(\neg 2)(\neg 3)$ respectively. Also denote $X=\bigcap\limits_{A\in S_p}\B(A,\dfrac{m_A}{p^c}), Y=\bigcap\limits_{A\in S_p}\B(A,\dfrac{m_A}{p})$.

$(\neg 1 \implies \neg 2):$ Take $f\in N(S)$ such that $f$ is not divisible by $p^c$. Let $h$ be the remainder of $f$ modulo $p^c$.
$\forall A\in S_p$, the remainder of $\dfrac{m_Af}{p^c}$ modulo $m_A$ is $\dfrac{m_Ah}{p^c}$, hence $\dfrac{m_Ah}{p^c}(A)=\dfrac{m_Af}{p^c}(A)=0, h\in\B(A,\dfrac{m_A}{p^c})$.
Thus $0\neq h\in X, X\neq\{0\}$.
    
$(\neg 2 \implies \neg 3):$ From $X\neq\{0\}$, we can take $0\neq f\in X$.
Since $\deg f<\deg p^c$, the multiplicity of $p$ as a prime factor of $f$ is at most $c-1$ (possibly 0),
so we can take $v\in\Z_+$ such that $p$ is a factor of $fp^v$ with multiplicity $c-1$.
Let $g$ be the remainder of $fp^v$ modulo $p^c$, then $p$ is also a factor of $g$ with multiplicity $c-1$,
from which it can be verified that $0\neq \dfrac{g}{p^{c - 1}}\in Y$.

$(\neg 3 \implies \neg 1):$ From $Y\neq\{0\}$, we can take $0\neq f\in Y$. Since $\deg f<\deg p$, $f$ and $p$ are coprime.
We prove that $\dfrac{f\mu_S}{p}\in N(S)$.
For $A\notin S_p, \dfrac{\mu_S}{p}(A)=0$; for $A\in S_p, \dfrac{fm_A}{p}(A)=0$.
Then from $\dfrac{\mu_S}{p}|\dfrac{f\mu_S}{p}$ and $\dfrac{fm_A}{p}|\dfrac{f\mu_S}{p}$, we have $\dfrac{f\mu_S}{p}\in N(S)$,
but $p^c$ does not divide $\dfrac{f\mu_S}{p}$.
\end{proof}

Combining with Proposition 1.5, the above proposition can be directly extended to a criterion for determining core sets:

\begin{Proposition} 

Assume $S$ is polynomial finite, and the standard factorization of $\mu_S$ is $\mu_S = \prod\limits^k_{i=1}p_i^{c_i}$.
Denote $S_p=\left\{A\in S\,\big|\,p^c|m_A\right\}$, then
$$S\text{ is a core set}\iff\forall 1\leq i\leq k,\bigcap_{A\in S_{p_i}}\B\left(A,\dfrac{m_A}{p_i}\right)=\{0\}.$$

In particular, if there exists $m\in F[x]$ such that $\varnothing\neq S\subseteq\C$, and let all prime factors of $\mu$ be $p_1,p_2,\cdots,p_k$, then
$$S\text{ is a core set}\iff\forall 1\leq i\leq k,\bigcap_{A\in S}\B\left(A,\dfrac{\mu}{p_i}\right)=\{0\}.$$
\end{Proposition} 

\begin{proof}

By Proposition 1.5, $S$ is a core set if and only if for $1\leq i\leq k$, $\forall f\in N(S), p^c_i|f$. Using Proposition 3.3 gives the proof.
\end{proof}

This result transforms the determination of core sets into testing the intersection of $\B(A,\dfrac{\mu}{p_i})$, providing a criterion for concretely analyzing core sets in different similarity classes. The subsequent parts of this paper will separately analyze linear and higher-degree factors of the minimal polynomial, giving criteria and size estimates for core sets.

\section{Analysis of Linear Factors}

In this section, we handle the case satisfying the following convention:
\begin{Convention}

Let $\C\subseteq M_n(F),\varnothing\neq S\subseteq\C, a\in F, x-a\mid\mu$.
\end{Convention}

In this case, nonzero polynomials in $\B(A,f_a)$ have degree $0$,
i.e., $\B(A,f_a)$ can be naturally regarded as a subset of $M_n(F)$ (i.e., a set of matrices).
This section will use linear algebra tools,
to characterize the structure of such $\B(A,f_a)$, and further transform the criterion given by Proposition 3.4, establishing a linear space description of core sets.

\begin{Lemma}

Let $\varnothing\neq S\subseteq \C, A\in S$, then

(1) $\bigcap\limits_{A\in S}\B(A,f_a)=\{0\}\iff\bigcap\limits_{A\in S}V_{\B(A,f_a)}=0$.

(2) $V_{\B(A,f_a)}=(\Im f_a(A))^\perp$.
\end{Lemma}

\begin{proof}

(1) ``$\Longrightarrow$": Let $v\in\bigcap\limits_{A\in S}V_{\B(A,f_a)}$ be a row vector.
For any $A\in S$, $v$ is a linear combination of row vectors of some matrices in $\B(A,f_a)$.
Assume these matrices are $P_1,P_2,\cdots,P_k$.
It is easy to verify that $\B(A,f_a)$ is a left ideal of $M_n(F)$.
Consider the matrix $H$ whose first row is $v$ and all other rows are zero,
Clearly $H$ is a matrix coefficient linear combination of $P_1,P_2,\cdots,P_k$.
This shows $H\in\B(A,f_a)$. By the arbitrariness of $A$ and the assumption, we have $v=0$.

``$\Longleftarrow$": First, $0\in\bigcap\limits_{A\in S}\B(A,f_a)$.
Assume by contradiction that $0\neq B\in\bigcap\limits_{A\in S}\B(A,f_a)$,
then $\{0,B\}\subseteq\bigcap\limits_{A\in S}\B(A,f_a)$,
By Proposition 2.5, $0\neq V_{\{0,B\}}\subseteq\bigcap\limits_{A\in S}V_{\B(A,f_a)}$, which is a contradiction.

(2) Let $u\in(\Im f_a(A))^\perp$. Consider the matrix $H$ whose first row is $u^T$ and all other rows are zero, clearly $H\in\B(A,f_a)$,
hence $u\in V_{\B(A,f_a)}$, i.e., $(\Im f_a(A))^\perp\subseteq V_{\B(A,f_a)}$.
Let $u\in V_{\B(A,f_a)}$, similarly consider the matrix $H$ whose first row is $u^T$ and all other rows are zero.
Similar to the discussion in (1), we know $H\in\B(A,f_a)$, this shows $u\in(\Im f_a(A))^\perp$, i.e., $V_{\B(A,f_a)}\subseteq(\Im f_a(A))^\perp$.

In summary, $V_{\B(A,f_a)}=(\Im f_a(A))^\perp$.
\end{proof}

\begin{Theorem}

Let $\varnothing\neq S\subseteq \C$, then

(1) If $S$ is a core set, then for any $a\in F$ satisfying $x-a\mid\mu$,
we have $\sum\limits_{A\in S}\Im f_a(A) = F^n$.

(2) When $\mu$ has only linear irreducible factors, the converse of the above statement also holds.
\end{Theorem}

\begin{proof}

By Proposition 3.4, if $S$ is a core set, then for any $a\in F$ satisfying $x-a\mid\mu$, we have $\bigcap\limits_{A\in S}\B(A,f_a)=\{0\}$,
When $\mu_\C$ has only linear irreducible factors, the converse also holds.
By Lemma 4.2, 
$$\bigcap_{A\in S}\B(A,f_a)=\{0\}\iff\sum_{A\in S}\Im f_a(A)=\left(\bigcap_{A\in S}V_{\B(A,f_a)}\right)^\perp=0^\perp=F^n.$$

The theorem is proved.
\end{proof}

To further analyze core sets in $\C$, it is necessary to classify the matrices in $\C$ more finely according to $\Im f_a(A)$.

\begin{Definition}

Consider a similarity $\C\subseteq M_n(F)$ and fix $a\in F$, where $x-a|\mu$. Define a equivalence relation: for $A,B\in\C$, if $\Im f_a(A)=\Im f_a(B)$, then $A,B$ are said to be equivalent. Denote the equivalence class of $A\in\C$ under this equivalence relation as $[A]$.
\end{Definition}

\begin{Lemma}

Let $\C\subseteq M_n(F)$ be a similarity, then

(1) $\forall A,B\in\C, \dim\Im f_a(A)=\dim\Im f_a(B)$.

(2) $\forall A,B\in\C, |[A]|=|[B]|$.

(3) Let $A\in\C$, 
$V$ be a subspace of $F^n$ satisfying $\dim V=\dim\Im f_a(A)$,
then $\exists B\in\C\st V=\Im f_a(B)$.
\end{Lemma}

\begin{proof}

(1) This is a fact from linear algebra.

(2) First, $|[A]|,|[B]|<\infty$. Let $Q\in\GL(n,F)$ satisfy $B=QAQ^{-1}$, denote $\sigma:x\mapsto Qx$.
Consider the mapping $\phi:[A]\to\C, U\to QUQ^{-1}$, clearly $\phi$ is injective.
Note that $\Im f_a(QUQ^{-1})=\sigma(\Im f_a(U))=\sigma(\Im f_a(A))=\Im f_a(B)$. Hence, $QUQ^{-1}\in[B]$.
Thus $|[A]|\leq|[B]|$, similarly $|[B]|\leq|[A]|$, so $|[A]|=|[B]|$.

(3) Since $\dim V=\dim\Im f_a(A)$, there exists a linear automorphism $\sigma$ on $F^n$ satisfying $\sigma(\Im f_a(A))=V$.
Let $Q\in\GL(n,F)$ satisfy $\sigma(x)=Qx$, then $V=\sigma(\Im f_a(A))=\Im f_a(QAQ^{-1})$,
Let $B=QAQ^{-1}$, then $V=\Im f_a(B)$.
\end{proof}

\begin{Convention}
For $A\in\C$, denote $r_a=\dim\Im f_a(A)$, abbreviated as $r$. The validation of this notation is guaranteed by Lemma 4.5(1).
\end{Convention}

Based on the above proposition, the enumeration of equivalence classes can be further characterized as follows:

\begin{Theorem}

 In this case, the equivalence relation in Definition 4.4 has $\prod\limits^r_{i=1}\dfrac{q^n-q^{i-1}}{q^r-q^{i-1}}$ equivalence classes, each of size
$|\C|\prod\limits^r_{i=1}\dfrac{q^r-q^{i-1}}{q^n-q^{i-1}}$.

\end{Theorem}

\begin{proof}

This is a direct consequence of Lemma 4.5.

\end{proof}

\section{Splitting of Higher-Degree Factors}

We now proceed to handle the case of higher-degree factors. In this and the next section, we adopt the following convention:

\begin{Convention}

Let $\varnothing\neq S\subset\C, |F|=q$ be prime, $\deg\mu=t$.
Assume that $m(x)=x^s+a_1x^{s-1}+\cdots+a_{s-1}x+a_s\in F[x]$
is a monic irreducible polynomial of degree $s\geq 2$ satisfying $m|\mu$.
In this case, $m$ is a separable polynomial.
Let $K$ be the splitting field of $m$ over $F$, and let the $s$ roots of $m$ in $K$ be $\alpha_1,\alpha_2,\cdots,\alpha_s$.
For notational convenience, let $\alpha=\alpha_1$.
\end{Convention}

We only discuss the case where $F$ is a finite field of prime order $q$ here, because in this case $F$ is a perfect field, and all irreducible polynomials over it are separable.

We know that $K$ is an $s$-dimensional $F$-linear space and $\{1,\alpha,\cdots,\alpha^{s-1}\}$ is a basis of $K$, so for $l \in K$ we have the representation
$$l=\sum_{i=0}^{s-1}b_i \alpha^i,b_i \in F,0\leq i\leq s-1,$$
which leads to the following lemma.

\begin{Lemma}

Let $g\in F[x], g|\mu$, then
$$\bigcap_{A\in S}\B_F(A,g)=\{0\}\iff\bigcap_{A\in S}\B_K(A,g)=\{0\}.$$
\end{Lemma}

\begin{proof}

We only need to prove ``$\Longrightarrow$".
Let $h(x)\in\bigcap\limits_{A\in S}\B_K(A,g)\subseteq M_n(K)[x]$, then $\forall A\in S,(hg)(A)=0$.
This induces $h(x)=\sum\limits_{i=0}^{s-1}\alpha^i h_i(x)$, where $h_i(x)\in M_n(F)[x],0\leq i\leq s-1$.
Note that $\deg h_i\leq\deg h<\deg\mu-\deg g$.
Then $0=(hg)(A)=\left(\left(\sum\limits_{i=0}^{s-1}\alpha^i h_i\right)g\right)(A)=\sum\limits_{i=0}^{s-1}\alpha^i (h_i g)(A)$,
hence $\forall 0\leq i\leq s-1,h_i\in\B_F(A,g)$, so $h_i=0$, i.e., $h=0$.
\end{proof}

We perform further transformation on the necessary and sufficient condition for core set determination:

\begin{Theorem}

$\bigcap\limits_{A\in S}\B_F(A,\dfrac{\mu}{m})=\{0\}\iff\forall 1\leq i\leq s,
\bigcap\limits_{A\in S}\B_K(A,f_{\alpha_i})=\{0\}.$
\end{Theorem}

\begin{proof}

By Lemma 5.2, it suffices to prove
$$\bigcap_{A\in S}\B_K(A,\dfrac{\mu}{m})=\{0\}\iff\forall 1\leq i\leq s,
\bigcap_{A\in S}\B_K(A,f_{\alpha_i})=\{0\}.$$

``$\Longrightarrow$": For any $1\leq i\leq s$,
assume that $0\neq B\in\bigcap\limits_{A\in S}\B_K(A,f_{\alpha_i})$,
then $0\neq \dfrac{Bm}{x-\alpha_i}\in\bigcap\limits_{A\in S}\B_K(A,\dfrac{\mu}{m})$, which is a contradiction.
Hence $\forall 1\leq i\leq s,\bigcap\limits_{A\in S}\B_K(A,f_{\alpha_i})=\{0\}$.

``$\Longleftarrow$": Let $h(x)=\sum\limits_{k=0}^{s-1}B_kx^k\in\bigcap\limits_{A\in S}\B_K(A,\dfrac{\mu}{m})$, i.e., $\forall A\in S,h(A)\dfrac{\mu}{m}(A)=0$.
For any $1\leq i\leq s$, from $\forall A\in S,\mu(A)=0$ and $x-\alpha_i|h(x)-h(\alpha_i)$ we have
$$0=h(A)f_{\alpha_i}(A)
=h(\alpha_i)f_{\alpha_i}(A),$$
hence $h(\alpha_i)\in\bigcap\limits_{A\in S}\B_K(A,f_{\alpha_i})$,
i.e., $h(\alpha_i)=0$.
But $\deg h<s$, so $h=0$, thus $\bigcap\limits_{A\in S}\B_K(A,\dfrac{\mu}{m})=\{0\}$. The theorem is proved.
\end{proof}

Further analysis of the properties of the splitting field yields the following lemma. With this lemma, we can further process Theorem 5.3.
\begin{Lemma}

$\forall 1\leq i,j\leq s,\bigcap\limits_{A\in S}\B_K(A,f_{\alpha_i})=\{0\}
\iff\bigcap\limits_{A\in S}\B_K(A,f_{\alpha_j})=\{0\}$.
\end{Lemma}

\begin{proof}

Since $m$ is irreducible, $\Gal(K/F)$ acts transitively on $\{\alpha_1,\cdots,\alpha_s\}$,
so $$\exists\sigma\in\Gal(K/F)\st\sigma(\alpha_i)=\alpha_j.$$
Naturally induce $\sigma$ to $M_n(K)[x]$, obtaining the isomorphism $\bigcap\limits_{A\in S}\B_K(A,f_{\alpha_i})\cong\bigcap\limits_{A\in S}\B_K(A,f_{\alpha_j})$.
The original proposition is proved.
\end{proof}

These results above ultimately lead to the following conclusion:
\begin{Theorem}

$\bigcap\limits_{A\in S}\B_F(A,\dfrac{f}{m})=\{0\}\iff
\bigcap\limits_{A\in S}\B_K(A,f_\alpha)=\{0\}\iff\sum\limits_{A\in S}\Im f_\alpha(A)=K^n.$
\end{Theorem}

Through the content of this section, we have transformed higher-degree polynomials into linear polynomials after field extension. In the next section, we will discuss the transformed proposition.

\section{Equivalence Class Partition in the Splitting Field}

Similar to the case of linear factors, we continue to study $\Im f_\alpha(A)$ and the equivalence classes defined in Definition 4.4. We restate the basic results in the extension field $K$. Here we first emphasize that Lemma 4.5(3) no longer holds in this case.

\begin{Lemma}

Let $\C\subseteq M_n(F)$ be a similarity class, then

(1) $\forall A,B\in\C,\dim\Im f_\alpha(A)=\dim\Im f_\alpha(B)$.

(2) $\forall A,B\in\C,|[A]|=|[B]|$.
\end{Lemma}

\begin{proof}

Similar to Lemma 4.5(1)(2).
\end{proof}

Since Lemma 4.5(3) no longer holds, there are additional issues. To further study this situation, this paper will use linear algebra methods to establish a deeper characterization of $\Im f_\alpha(A)$.

\begin{Theorem}
(1) Let $v^j_i,1\leq i\leq r,0\leq j\leq s-1$ be $rs$ linearly independent vectors in $F^n$. Let $v_i=\sum\limits^{s-1}_{j=0}\alpha^jv^j_i$, then $\exists A\in\C\st\Im f_\alpha(A)=\span\{v_1,\cdots,v_r\}$. In this case, $v_1,\cdots,v_r$ naturally form a basis.

(2) Let $V\in\C,\Im f_\alpha(A)$ have a basis $v_1,\cdots,v_r$. Let $v_i=\sum\limits^{s-1}_{j=0}\alpha^jv^j_i$, where $v^j_i\in F^n,1\leq i\leq r,0\leq j\leq s-1$, then the $v^j_i$ are linearly independent.
\end{Theorem}

\begin{proof}

Take the generalized Jordan canonical form $J_0$ of $\C$ and consider a Jordan block $J$ of $J_0$.
Let $m$ be a prime factor of $\mu$ with multiplicity $c$,then $J$ is a $sc\times sc$ matrix.
It is easy to see that $f_\alpha(J)\neq O$ if and only if $J$ corresponds to the elementary divisor $m^c$.
First, $f_\alpha(J)=\dfrac{m^c}{x-\alpha}(J)\dfrac{f}{m^c}(J)$, hence $\Im f_\alpha(J)=\Im \dfrac{m^c}{x-\alpha}(J)$.
Take $0\neq u=\sum\limits^{s-1}_{j=0}\alpha^ju^j\in K^{sc}$, where $u^j\in F^{sc},0\leq j\leq s-1$.
We claim that there exist $rs$ linearly independent vectors $u^j_i,1\leq i\leq r,0\leq j\leq s-1$ in $F^n$ satisfying $\Im f_\alpha(J_0)=\span\{u_1,\cdots,u_r\}$, which only requires proving that $u^j\in F^{sc},0\leq j\leq s-1$ are linearly independent.

First, $Ju=\alpha u$, i.e., $J\sum\limits^{s-1}_{j=0}\alpha^ju^j=\alpha\sum\limits^{s-1}_{j=0}\alpha^ju^j=\sum\limits^{s-1}_{j=0}\alpha^{j+1}u^j=\sum\limits^{s-2}_{j=0}\alpha^{j+1}u^j-\sum\limits^{s-1}_{j=0}a_{s-j}\alpha^ju_{s-1}$.
By the $F$-linear independence of $1,\alpha,\cdots,\alpha^{s-1}$, $\span\{u^0,\cdots,u^{s-1}\}$ is a $J$-invariant subspace,
This implies $s|\dim\span\{u^0,\cdots,u^{s-1}\}$.
Also, $u\neq 0$ and $\dim\span\{u^0,\cdots,u^{s-1}\}\leq s$, so $\dim\span\{u^0,\cdots,u^{s-1}\}=s$.
Thus the claim holds.

(1) Let $Q\in\GL(n,F)$ satisfy $\forall 1\leq i\leq r,0\leq j\leq s-1,v^i_j=Qu^i_j$.
Let $A=QJQ^{-1}$, then it satisfies the requirement.

(2) First prove that $J_0$ satisfies the theorem. Let $V$ be a subspace of $F^n$ satisfying $\Im f_\alpha(J_0)\subseteq V_K$.
By the definition of $V_K$, $\forall 1\leq i\leq r,0\leq j\leq s-1,u^i_j\in V$, hence $\dim V\geq rs$.
Assume that $v^j_i$ are linearly dependent. Let $W=\span\{v^j_i|1\leq i\leq r,0\leq j\leq s-1\}$,
then $\dim W<rs$ and $\Im f_\alpha(J_0)\subseteq W_K$, leading to a contradiction. Hence $J_0$ satisfies the theorem.
Let $Q\in\GL(n,F)$ satisfy $A=QJ_0Q^{-1}$, then $\Im f_\alpha(A)=Q\Im f_\alpha(J_0)$,
Here $Q^{-1}v_1,\cdots,Q^{-1}v_r$ form a basis of $\Im f_\alpha(J_0)$.
Also denote $Q^{-1}v_i=u_i=\sum\limits^{s-1}_{j=0}\alpha^ju^j_i$, we can see that $u^j_i$ are $F$-linear combinations of $v^j_i$.
Since $u^j_i$ are linearly independent, $v^j_i$ are linearly independent.
\end{proof}

We now establish the enumeration of equivalence classes for this case.

\begin{Theorem}
In this case, the equivalence relation in Definition 4.4 has $\dfrac{\prod\limits^{rs}_{i=1}(q^n-q^{i-1})}{\prod\limits^{r}_{i=1}(q^{rs}-q^{(i-1)s})}$ equivalence classes, each of size $\dfrac{|\C|\prod\limits^{r}_{i=1}(q^{rs}-q^{(i-1)s})}{\prod\limits^{rs}_{i=1}(q^n-q^{i-1})}$.

\end{Theorem}
\begin{proof}
 According to Lemma 6.1 and Theorem 6.2, we calculate the number of equivalence classes as
$\dfrac{\prod\limits^{rs}_{i=1}(q^n-q^{i-1})}{\prod\limits^{r}_{i=1}(q^{rs}-q^{(i-1)s})}$.
\end{proof}
Referring to Theorem 4.7, we can see that it essentially completes the missing case of $s=1$ in Theorem 6.3. This implies that Theorem 6.3 holds for any positive integer $s$. Hence, We can provide an upper bound estimate for the size of non-core sets:

\begin{Theorem}
Let $\varnothing\neq S\subseteq\C$ be a non-core set, then $|S|\leq\dfrac{4|\C|}{q}$.
\end{Theorem}

\begin{proof}
According to Proposition 3.4, there exists an irreducible factor $p$ of $\mu_\C$ such that $$\bigcap_{A\in S}\B\left(A,\dfrac{\mu_\C}{p}\right) \neq \{0\}.$$
Let $K$ be the splitting field of $p$ over $F$ (if $\deg p=1$ then take $F=K$).
Let $\alpha \in K$ be a root of $p$, and denote $s=\deg p$.
For $A\in\C$, let $r=\dim\Im_K f_a(A)$.
By Theorem 4.3 and Theorem 5.5, there exists an $(n-1)$-dimensional subspace $V$ of $K^n$ such that $\forall A\in S,\Im f_\alpha(A)\subseteq V$.

Then using Lemma 6.1 and the counting result from Theorem 6.3, we have
\begin{equation*}
\begin{aligned}
|S|&\leq\dfrac{|\C|\prod\limits^{r}_{i=1}(q^{rs}-q^{(i-1)s})}{\prod\limits^{rs}_{i=1}(q^n-q^{i-1})}\cdot\prod\limits^{r}_{i=1}\dfrac{q^{(n-1)s}-q^{(i-1)s}}{q^{rs}-q^{(i-1)s}}=\dfrac{|\C|\prod\limits^{r}_{i=1}(q^{(n-1)s}-q^{(i-1)s})}{\prod\limits^{rs}_{i=1}(q^n-q^{i-1})}
\\&\leq\dfrac{|\C|q^{(n-1)rs}}{q^{nrs}\prod\limits^{\infty}_{i=1}(1-q^{-i})}\leq\dfrac{4|\C|}{q^{rs}}
\leq \dfrac{4|\C|}{q}.\qed
\end{aligned}
\end{equation*}
\end{proof}

\section{Asymptotic Proportion of Pure Core Sets}

To discuss asymptotic results, we first give a lemma.

\begin{Lemma}

Let $c\in(0,1)$, then the proportion of subsets of an $N$-element set with size at most $\dfrac{N}{2}(1-c)$ is at most $e^{-\frac{c^2N}{4}}$.
\end{Lemma}

\begin{proof}

Without loss of generality, let $S=\{i\in\Z|1\leq i\leq N\}$ be an $N$-element set. Using the probabilistic method, let $D\subseteq S$ be a random subset.
Let random variables $X_i=\begin{cases}1 & i\in D,\\0 & i\notin D,\end{cases}$ and $X=\sum\limits^N_{i=1}X_i$.
We have $\mathbb{P}(X_i=1)=\dfrac{1}{2},1\leq i\leq N$ and $X_i,1\leq i\leq N$ are mutually independent,
Also $X=|S|,\E(X)=\dfrac{n}{2}$.
By Chernoff's inequality, $\mathbb{P}(|S|=X\leq\dfrac{N}{2}(1-c))\leq e^{-\frac{c^2N}{4}}$.
\end{proof}

\begin{Theorem}
Let $S\subseteq M_n(\Fq)$ be a random subset, 
then$$\lim\limits_{q\to\infty}\mathbb{P}(S\text{ is a pure core set})=\lim\limits_{q\to\infty}\mathbb{P}(S\text{ is a core set})=1.$$
\end{Theorem}

\begin{proof}

It suffices to prove $\lim\limits_{q\to\infty}\mathbb{P}(S\text{ is a pure core set})=1$. Without loss of generality, assume $q>10$.
By Theorem 6.4, for a non-core set $\varnothing\neq S\subseteq\C$, we have $|S|\leq\dfrac{2|\C|}{5}$.

First, the number of similarity classes is at most $q^{n^2}$, and the intersections of $S$ with different similarity classes are independent with respect to the core property.
Take any similarity class $\C\subseteq M_n(\Fq)$ and random subset $S\subseteq\C$, we only need to consider the case $\deg\mu\geq 2$.

In this case, $|\C|\geq\dfrac{\prod\limits^{rs}_{i=1}(q^n-q^{i-1})}{\prod\limits^{r}_{i=1}(q^{rs}-q^{(i-1)s})}\geq\dfrac{(q^n-q^{n-1})^{rs}}{q^{r^2 s}}\geq q-1$.
By Lemma 7.1, $$\mathbb{P} (S\text{ is not a core set})\leq\mathbb{P}(S\leq\dfrac{2|\C|}{5})\leq e^{-\frac{|C|}{100}}\leq e^{-\frac{q-1}{100}}\leq e^{-\frac{q}{200}}.$$
Returning to the original assumption, $\mathbb{P}(S\text{ is a pure core set})\geq(1-e^{-\frac{q}{200}})^{q^{n^2}}$.
Note that $\lim\limits_{q\to\infty}(1-e^{-\frac{q}{200}})^{q^{n^2}}=1$,
hence $\lim\limits_{q\to\infty}\mathbb{P}(S\text{ is a pure core set})=1$.
The theorem is proved.
\end{proof}

{\small

\textsc{Department of Mathematical Sciences}, Tsinghua University, Beijing, 100084, P. R. China.

\textit{Email address}: wanghong24@mails.tsinghua.edu.cn 

\textsc{Zhili College}, Tsinghua University, Beijing, 100084, P. R. China.

\textit{Email address}: 
yizhi-zh24@mails.tsinghua.edu.cn

}
\end{document}